\documentclass[a4paper,12pt, reqno]{article}
\usepackage{amsfonts}
\usepackage{amsmath}

\usepackage{amsfonts,amssymb,amsmath,amsthm}
\usepackage{enumerate}

\newtheorem{theorem}{Theorem}[section]
\newtheorem{proposition}[theorem]{Proposition}

\newtheorem{lemma}[theorem]{Lemma}
\theoremstyle{remark}
\theoremstyle{definition}
\newtheorem{remark}[theorem]{Remark}

\newtheorem{definition}[theorem]{Definition}

\numberwithin{equation}{section}

\newcommand{\be} {\begin{equation}}
\newcommand{\ee} {\end{equation}}
\newcommand{\bea} {\begin{eqnarray}}
\newcommand{\eea} {\end{eqnarray}}
\newcommand{\Bea} {\begin{eqnarray*}}
\newcommand{\Eea} {\end{eqnarray*}}

\newcommand{\la} {\lambda}

\newcommand{\La} {\Lambda}

\newcommand{\Om} {\Omega}
\newcommand{\Mp}{\mathcal M^+}
\newcommand{\Mm}{\mathcal M^-}
\newcommand{\Rn}{\mathbb{ R}^{n}}

\def\XXint#1#2#3{{\setbox0=\hbox{$#1{#2#3}{\int}$}
    \vcenter{\hbox{$#2#3$}}\kern-.5\wd0}}

\title{Overdetermined problems for the normalized $p$-Laplacian}

\author{
 {\sc  Agnid Banerjee}   \and  {\sc Bernd Kawohl} 
 }

\date{\today}

\begin{document}

\maketitle

\begin{abstract}
We extend the symmetry result of Serrin \cite{S} and Weinberger \cite{W} from the Laplacian operator to the highly degenerate game-theoretic $p$-Laplacian operator and show that viscosity solutions of $-\Delta_p^Nu=1$ in $\Omega$, $u=0$ and $\tfrac{\partial u}{\partial\nu}=-c\neq 0$ on $\partial\Omega$ can only exist on a bounded domain $\Omega$ if $\Omega$ is a ball.

\end{abstract}

\bigskip


 
 %%%%%%%%%%%%%%%%%%%%%%%%%%%%%%%%%%%%%%%%%%%%%%%%%%%%%%%%
\section{Introduction}\label{introduction}
%%%%%%%%%%%%%%%%%%%%%%%%%%%%%%%%%%%%%%%%%%%%%%%%%%%%%%%%
In a seminal paper \cite{S}  Serrin showed that the following overdetermined boundary problem can only have a solution $u\in C^2(\overline\Omega)$ if $\Omega$ is a ball.
\begin{equation}\label{linear}
-\Delta u=1\quad\hbox{ in }\Omega,\qquad\qquad u=0\quad\hbox{ and }\quad\frac{\partial u}{\partial \nu}=c<0\hbox{ on }\partial\Omega
\end{equation}
Here $c$ is constant and $\Omega\subset{\mathbb R}^n$ is a bounded connected domain with boundary of class $C^{2}$. Serrin used Alexandrov's moving plane method for his proof, while Weinberger \cite{W} found a proof using Rellich's identity and the fact that a related function $P(x)=|\nabla u|^2+\tfrac{2}{n}u$ is constant in $\Omega$. Only the second method of proof has been adapted to a situation where the Laplacian operator is replaced by the $p$-Laplacian in \cite{GL} and \cite{K0}.

\smallskip
In this paper we treat the case that the Laplacian is replaced by the normalized or game-theoretic $p$-Laplacian $\Delta_p^N$ which is defined for any $p\in(1,\infty)$ by
\begin{equation}\label{normalizeddef}
\Delta_p^Nu:=\frac{1}{p}|\nabla u|^{2-p}{\rm div}\left( |\nabla u|^{p-2}\nabla u\right)=\frac{1}{p}\Delta_1^Nu+\frac{p-1}{p}\Delta_\infty^Nu,
\end{equation}
a convex combination of the limiting operators
\begin{equation}
\Delta_1^Nu:=|\nabla u|{\rm  div}\left( \frac{\nabla u}{|\nabla u|}\right)\quad\hbox{ and }\quad\Delta_\infty^Nu:=\frac{\sum_{i,j=1}^nu_{x_i}u_{x_i x_j}u_{x_j}}{|\nabla u|^2}.
\end{equation}
Note that this operator is not in divergence form. Therefore it resists attempts to treat it with variational methods. On the other hand it is quite benign, because its coefficient matrix is bounded from below by $\min\{\tfrac{1}{p},\tfrac{p-1}{p}\}I$ and from above by $\max\{\tfrac{1}{p},\tfrac{p-1}{p}\}I$. Therefore the moving plane method seems more appropriate in this context.

Note also, that the above definition of the normalized $p$-Laplacian needs further explanation when $\nabla u=0$. The definition  of and a weak comparison principle for continuous viscosity solutions are given below. These and an existence and uniqueness result can be found for instance in \cite{LW} or \cite{Ku}. Our main result answers an open problem from \cite{K1}. 

%\textbf{now in the theorem, we only assume $u$ is continuous because of the regularity result that has been put in the Appendix}
\begin{theorem} \label{Main}
For $p\in(1,\infty)$ let $u\in C(\overline\Omega)$ be a viscosity solution to the overdetermined boundary value problem
\begin{equation}
-\Delta_p^N u=1\quad\hbox{ in }\Omega,\qquad\qquad u=0\quad\hbox{ and }\quad\frac{\partial u}{\partial \nu}=c<0\hbox{ on }\partial\Omega
\end{equation}
on a connected bounded domain $\Omega$ with with boundary of class $C^{2}$. Then $\Omega$ must be a ball.
\end{theorem}
\begin{remark}
We note that the Neumann condition $ \ \frac{\partial u}{\partial \nu}=c<0\hbox{ on }\partial\Omega$ is interpreted  in the following sense: Any $C^{2}$ function  $\varphi$ such that $\varphi-u$ has a minimum at a point $x \in \partial \Omega$ satisfies $\ \frac{\partial \varphi}{\partial \nu}(x) \leq c$ at $x$. Similarly,  any $C^{2}$ function $\psi$ such that $\psi-u$ has a maximum at  a  point $y \in \partial \Omega$ satisfies $\ \frac{\partial \psi}{\partial \nu}(y) \geq c$.\end{remark}

\begin{remark}
It was pointed out in  \cite{K1} that Theorem \ref{Main} remains true for $p=1$, while for $p=\infty$ it is generally false.

In fact for $p=1$ the equation can be rewritten as $-\tfrac{n-1}{p}H(x) u_\nu(x)=$, where $H(x)$ denotes mean curvature of the level set passing through $x$, and in view of the constant Neumann data this means that $\partial\Omega$ has constant mean curvature. Therefore $\Omega$ is a ball of radius $\tfrac{1-n}{p}c$. 

As explained in \cite{BK}, for $p=\infty$ the right $P$-function is $|\nabla u|^2+2u$, and annuli are cases in which the overdetermined problem has viscosity solutions of class $C^1$. The case $p=\infty$ was also studied in great detail in a series of papers by Crasta and Fragal\'a, who relaxed the $C^2$ smoothness of the boundary, see e.g. \cite{CF}.

\end{remark}
\smallskip

The normalized $p$-Laplacian has also been studied in the context of evolution equations in a number of papers, see \cite{JK, D, BG1, BG2, PR, Ju, JS}.
%%%%%%%%%%%%%%%%%%%%%%%%%%%%%%%%%%%%%%%%%%%%%%%%%%%%%%%%%%%
\section{Definitions and Comparison Result}
In the notation of the theory of viscosity solutions we study the equation
\begin{equation}\label{F_p}
F_p(\nabla u,\nabla^2u)= -\frac{p-2}2 |\nabla u|^{-2}\langle \nabla^2u \nabla u, \nabla u\rangle-\frac{1}{p}{\rm trace}\nabla^2u-1=0.
\end{equation}\
\begin{definition}\label{vs}Following \cite{CIL},
$u\in C(\Omega)$ is a {\bf viscosity solution} of the equation $F(\nabla u,\nabla^2u)=0$, if it is both a viscosity subsolution and a viscosity supersolution.
\medskip

$u$ is a {\bf viscosity  subsolution} of $F(\nabla u,\nabla^2u)=0$, if for every $x\in\Omega$ and $\varphi\in C^2$ such that
$\varphi-u$ has a minimum at $x$, the inequality $F_*(\nabla \varphi,\nabla^2\varphi)\leq0$ holds. Here $F_*$ is the lower semicontinuous hull of $F$.
\medskip

$u$ is a {\bf viscosity supersolution}  of $F(\nabla u,\nabla^2u)=0$, if for every $x\in\Omega$ and $\psi\in C^2$ such that
$\psi-u$ has a maximum at $x$, the inequality $F^*(\nabla \psi,\nabla^2\psi)\geq0$ holds. Here $F^*$ is the upper semicontinuous hull of $F$.

\end{definition}
If $X$ denotes a symmetric real valued matrix, we denote its eigenvalues by $\lambda_{min}=\lambda_1\leq \lambda_2\leq\ldots\leq\lambda_n=\lambda_{max}$. Using this notation, it is   a simple exercise to find out that 
\begin{equation}
F_*(q,X)=\begin{cases}  \qquad F(q,X)\quad&\hbox{if }q\not=0,\\
\inf_{a\in{\mathbb R}^n\setminus\{ 0\}}F(a,X)\quad&\hbox{if }q=0,      
\end{cases}\end{equation}
so 
\begin{equation}
F_*(0,X)=\begin{cases}-\tfrac{p-1}{p}\lambda_1-\tfrac{1}{p}\sum_{i=2}^{n}\lambda_i-1\quad\hbox{ for }p\in(1,2],\\
-\tfrac{p-1}{p}\lambda_n-\tfrac{1}{p}\sum_{i=1}^{n-1}\lambda_i-1\quad\hbox{ for }p\in[2,\infty),\end{cases}
\end{equation}
while
\begin{equation}
F^*(q,X)=\begin{cases} \qquad F(q,X)\quad&\hbox{if }q\not=0,\\
\sup_{a\in{\mathbb R}^n\setminus\{ 0\}}F(a,X)\quad&\hbox{if }q=0,      
\end{cases}
\end{equation}
that is
\begin{equation}
F^*(0,X)=\begin{cases}-\tfrac{p-1}{p}\lambda_n-\tfrac{1}{p}\sum_{i=1}^{n-1}\lambda_i-1\quad\hbox{ for }p\in(1,2],\\
-\tfrac{p-1}{p}\lambda_1-\tfrac{1}{p}\sum_{i=2}^{n}\lambda_i-1\quad\hbox{ for }p\in[2,\infty).\end{cases}
\end{equation}

\medskip
The following comparison principle has been derived in \cite{LW,Ku}.
\begin{proposition}\label{comp}
Suppose $u$ and $v$ are in $C(\overline D) $ and are viscosity super- resp. subsolutions of $F_p=0$ on a domain $D$ and $u\geq v$ on $\partial D$. Then $u\geq v$ in $D$.
\end{proposition}

 It can be used to show the positivity of $u$ and a Hopf Lemma.
%\textbf{We can get rid of this Hopf Lemma since we only need Hopf Lemma for linearized PDE satisfied by $w=u-v$  near the boundary. I am putting some explanation in the proof of the main result  and you can see if you agree with me or not. Finally you can decide what to keep}
\begin{lemma}\label{Hopf}
Suppose $\Omega$ satisfies a uniform interior sphere condition and $u\in C(\overline{\Omega})$ is a viscosity solution of $F_p(\nabla u,\nabla^2u)=0$ in $\Omega$ such that $u=0$ on $\partial \Om$. Then $u$ is positive in $\Omega$ and there exists a number $a>0$ such that for all $y \in \partial \Om$  
$$\limsup_{t \to 0^{+}}\frac{u(y)-u(y- t \nu(y))}{t} \leq -a<0.$$
\end{lemma}
Here $\nu(y)$ denotes the outward  unit normal at $y$. 
In fact, one can compare $u$ to a radially symmetric and radially decreasing solution $v$ on the interior of the sphere. On a ball solutions $v$ of the Dirichlet problem for $F_p=0$ are unique by Proposition \ref{comp}, so they are  necessarily radial. In polar coordinates $F_p=0$ turns into the tractable ODE \cite{KKK}
$$-\tfrac{p-1}p v_{rr}-\tfrac{n-1}{pr}v_r=1 \quad \text{in}\ (0,R)$$ with $v_r(0)=0=v(R)$ as boundary conditions, and this boundary value problem has the explicit solution $$v(r)=\frac{p}{2(p+n-2)}(R^2-r^2),$$ so that $u$ is positive in every ball with radius $R$ contained in $\Omega$. Moreover, Lemma \ref{Hopf} holds with $a=\tfrac{Rp}{p+n-2}$.\hfill\qed

%%%%%%%%%%%%%%%%%%%%%%%%%%%%%%%%%%%%%%%%%%%%%%%%%%%%%%%%%%%%
\section{Proof of Main Result}
 To prove Theorem \ref{Main} we follow an idea developed in \cite{AG}.  {We first note that from the regularity result stated in  Theorem \ref{imp2} in the Appendix, we have that $u$ is  in $C^{1, \beta} (\overline{\Om})$ for some $\beta=\beta(n, p, \Om)$ and therefore the Neumann condition is realized  in the classical pointwise sense.  Now  because by assumption $|\nabla u|=-c>0$ on $\partial\Omega$ and $u\in C^{1, \beta} (\overline\Omega)$, we have that  $|\nabla u|>0$ in an $\varepsilon$-neighborhood $S_\varepsilon$ of $\partial\Omega$ inside $\Omega$ defined by $S_\varepsilon:=\{ x\in\overline\Omega\ ; \ d(x,\partial\Omega)<\varepsilon\}$. Therefore  the operator $F_p$ is well-defined in the classical sense in $S_\varepsilon$}.  {Moreover in $S_\varepsilon$, since  $|\nabla u| > 0$, we have that  $u$ solves} 
 \[
 \Sigma_{i, j=1}^n a_{ij}(x) u_{x_i x_j}= -1 
 \]
 where
 \[
 a_{ij}(x)= \frac{1}{p} (\delta_{ij} + (p-2)\frac{u_{x_i} u_{x_j}}{|\nabla u|^2})
 \]
{is uniformly  elliptic and is in $C^{\beta}(S_\varepsilon)$.  Consequently,  by the classical Schauder theory  we can assert that   $u$ is of class $C^{2, \beta}_{loc}$ in $S_\varepsilon$.   We now move a  hyperplane}, say $T_\lambda:=\{x\in{\mathbb R}^n\ |\ x_1=\lambda\}$ from the left by the amount $\varepsilon/2$ in $x_1$-direction into $\Omega$ and compare the original solution $u(x)$ to the reflected one $v(x)=u(x-2\lambda e_1)$ in the reflected cap. By the weak comparison principle, Proposition \ref{comp}, we know that $u\geq v$ in the reflected cap $\Sigma_\lambda'$.  Moreover,  since $|\nabla u|, |\nabla v|>0$ in $S_\varepsilon$, we have that $u,v$ solve an  equation in $S_\varepsilon$ of the form 
\begin{equation}\label{e10}
\tilde{F}(\nabla h, \nabla^2 h)=0
\end{equation}
where $\tilde F$ is uniformly elliptic and smooth in its arguments.  Therefore $w=u-v$ solves the following equation linearized equation in $S_\varepsilon$,
\begin{equation}\label{e11}
\Sigma_{i,j=1}^n c_{ij} w_{x_i x_j} + \langle b, \nabla w\rangle =0,
\end{equation}
where 
\[
c_{ij}= \int_{0}^{1} \frac{ \partial \tilde F}{\partial m_{ij}} ( t\nabla u + (1-t) \nabla v, t \nabla^2 u + (1-t) \nabla^2 v) dt
\]
and
\[
b_i= \int_{0}^{1} \frac{\partial \tilde F}{\partial p_i} ( t \nabla u + (1-t) \nabla v, t \nabla^2 u +(1- t) \nabla^2 v) dt .
\]
Moreover $[c_{ij}]$ is  uniformly elliptic and the first order coefficients $b_{i}$ are bounded  in $S_\varepsilon$.  Note that over here, we think of $\tilde F: \mathbb{R}^{n} \times \mathbb{R}^{n^2} \to \mathbb{R}$ as  a function of the matrix $[m_{ij}] \in \mathbb{R}^{n^2}$ and the vector $p=(p_1, ....., p_n) \in \Rn$. Therefore,  $\frac{ \partial \tilde F}{\partial m_{ij}}$ is to be thought of as the partial derivative of $ \tilde F$ with  respect to the coordinate $m_{ij}$ in $\mathbb{R}^{n^2}$ and $\frac{\partial \tilde  F}{\partial p_i}$ is  the partial derivative  with respect to the coordinate $p_i$ in $\Rn$. 

\medskip

Since $w$ solves the uniformly elliptic PDE \eqref{e11} in $\Sigma_\lambda'\cap S_\varepsilon$,   by  the classical strong maximum principle applied to  $w$, we get that $u>v$ in $\Sigma_\lambda'\cap S_\varepsilon$ and $w_{x_1}>0$ on the plane $T_\lambda\cap S_{\varepsilon/2}$. The latter inequality follows from the classical Hopf Lemma applied to $w$.
We continue to move the hyperplane across. Even if the operator might become degenerate because we pass a critical point of $u$, the weak comparison principle continues to hold, so that $u\geq v$ in the reflected cap, until one of the following cases occurs.

$\quad$ i) The hyperplane and $\partial\Omega$ meet under a right angle in a point $P$.

$\quad$ ii) The reflected cap touches $\partial\Omega$ from the inside of $\Omega$ in a point $Q$.

In case i) we can apply the strong maximum principle again and conclude that either $u>v$ in the reflected cap intersected with an $\varepsilon/2$ neighborhood of the point $P$, or $u\equiv v$ there. But by Serrin's corner lemma applied to $w=u-v$ which solves \eqref{e11},  the first case $w>0$ is ruled out. In fact $\nabla u(P)=\nabla v(P)$ because the normal derivatives coincide there and the tangential ones vanish.  So a partial derivative of $w$ in any direction $\eta$ must vanish there. Let us see what happens to second partial derivatives in direction $\eta=\alpha \nu +\beta \tau$, where $\tau$ is a unit vector tangent to $\partial\Omega$ at $P$. We claim that
\begin{equation}\label{ident1}
u_{\eta\eta}(P)=\alpha^2 u_{\nu\nu}(P)+\alpha\beta u_{\nu\tau}(P)+\alpha\beta u_{\tau\nu}(P)+\beta^2u_{\tau\tau}(P)=v_{\eta\eta}(P).
\end{equation}
In fact in $P$ one can rewrite the differential equations for $u$ and $v$ as (see \cite{K1})
\begin{equation}
\tfrac{p-1}{p}u_{\nu\nu}(P)+\tfrac{n-1}{p}H(P)u_\nu(P)=-1=\tfrac{p-1}{p}v_{\nu\nu}(P)+\tfrac{n-1}{p}H(P)v_\nu(P)
\end{equation}
where $H$ is the mean curvature of the $C^2$ boundary. Since $u_\nu(P)=c=v_\nu(P)$ we conclude that $u_{\nu\nu}(P)=v_{\nu\nu}(P)$. For the same reason $(u_\nu)_\tau=0=(v_\nu)_\tau$. Now $(u_\tau)_\nu={u_\nu}_\tau-\kappa u_\nu$, where $\kappa$ denotes the curvature of $\partial\Omega$ in direction $\tau$, so that   $(u_\tau)_\nu(P)=(v_\tau)_\nu(P)$. Finally one can observe that $u_{\tau\tau}(P)=u_{ss}(P)+\kappa(P)u_\nu(P)$, where $s$ denotes arclength along the curve that is cut out of $\partial\Omega$  by the plane spanned by $\tau$ and $\nu$. Since $u$ and $v$ are constant on $\partial\Omega$, we have $u_{ss}(P)=v_{ss}(P)$, and since $\nabla u(P)=\nabla v(P)$, this completes the proof of (\ref{ident1}). Therefore we  can conclude that $w_\eta(P)=w_{\eta\eta}(P)=0$ and at this point,  Serrin's corner lemma implies that $w\equiv 0$ in $\Sigma_\lambda'\cap S_\varepsilon$.

\smallskip
In case ii) we can also conclude that either $u>v$ in the reflected cap intersected with an $\varepsilon/2$ neighborhood of the point $Q$ or $u\equiv v$. Now in former case, i.e. when $u>v$ in the reflected  cap, we recall again that $w$ solves the uniformly elliptic PDE \eqref{e11} in $S_\varepsilon$.But then by the  Hopf's Lemma  applied to $w$, we get that $\partial w/\partial\nu<0$ at $Q$, which contradicts the fact that $u$ satisfies constant Neumann data.  Consequently, we have that  $w\equiv 0$ in $\Sigma_\lambda'\cap S_\varepsilon$.

In both cases $\partial\Omega$ and $u$ are locally Steiner-symmetric in direction $x_1$. To see that they are also globally symmetric, one can argue as follows. For reasons of continuity the set of points in which the reflected boundary coincides with the original boundary is closed in $\partial\Omega$. But it is also open in $\partial \Omega$. In fact in any point that belongs to the boundary of this intersection one can apply the corner lemma again to see that a whole neighborhood still belongs to it.

Since this Steiner symmetry happens in any direction, we can conclude that  $\Omega$ is a ball and $u$  is radial and radially decreasing.   \hfill\qed

\section{Appendix}
In this section, we state and prove  a basic regularity result which has been referred to  in the proof of Theorem \ref{Main} in the previous section.   In order to do so, we  first introduce the relevant notion of extremal \emph{Pucci} type operators. 
Let $\Mp$ and $\Mm$ denote the maximal and minimal \emph{Pucci operators} corresponding to $\lambda,  \Lambda$, i.e., for every $M\in \mathcal S_n$ we have
\[
\Mp(M) = \Mp(M,\lambda,\Lambda) = \La \sum_{e_i>0} e_i + \la \sum_{e_i<0} e_i,
\]
\[
\Mm(M) = \Mm(M,\lambda,\Lambda) = \la \sum_{e_i>0} e_i + \La \sum_{e_i<0} e_i,
\]
where $e_i = e_i(M)$ indicate the eigenvalues of $M$. Hereafter, the dependence of $\Mp$ and $\Mm$ on $\lambda, \Lambda$  will be suppressed. As is well known, $\Mp$ and $\Mm$ are uniformly elliptic fully nonlinear operators. 

\medskip

We now state the following important lemma which connects the normalized $p$ Laplacian operator $\Delta_p^N$ to  appropriate  maximal and minimal \emph{Pucci operators}. 

\begin{lemma}\label{imp1}\
Let $u\in C(\Omega)$ be a viscosity solution to
\begin{equation}\label{E0}
\Delta_p^N u= f\ \text{in $\Omega$}
\end{equation}
and let $f$ be bounded. Then $u$ satisfies the following differential  inequalities in the viscosity sense
\begin{equation}\label{maine}
\Mp(\nabla^2 u) + K \geq 0 \geq \Mm(\nabla^2 u) -  K
\end{equation}
where $\Mp, \Mm$ are the pair of extremal \emph{Pucci operators} corresponding to $\la= \min\{\tfrac{1}{p},\tfrac{p-1}{p}\}$ and $\La= \max\{\tfrac{1}{p},\tfrac{p-1}{p}\}$ and $K= ||f||_{L^{\infty}(\Om)}$. 
\end{lemma}
\begin{proof}
The proof is a straightforward consequence of the fact that the    coefficient matrix corresponding to $\Delta_p^N$ is bounded  from below by $\min\{\tfrac{1}{p},\tfrac{p-1}{p}\}I$ and from above by $\max\{\tfrac{1}{p},\tfrac{p-1}{p}\}I$.
\end{proof}
We now state our main result  in this section concerning  the regularity of viscosity solutions to the equation in Theorem \ref{Main} up to the boundary.
\begin{theorem}\label{imp2}
Let $\Omega$ be a $C^{2}$ domain and  $x_0 \in \partial \Omega$. Let  $u$ be a viscosity solution to
\begin{equation}\label{e0}
\begin{cases}
-\Delta_p^N u= f\ \text{in}\ \Omega \cap B_{2r}(x_0)
\\
u=g\ \text{on}\  \partial \Omega \cap B_{2r}(x_0)
\end{cases}
\end{equation}
where $f \in C(\overline{B_{2r} (x_0) \cap \Omega})$ and $g \in  C^{1, \alpha}( \partial \Omega \cap B_{2r}(x_0))$. Then $u \in C^{1, \beta}(\overline{\Omega \cap B_{r}(x_0)})$ where $\beta=\beta(\alpha, \Omega, f, g)$. 
\end{theorem}

\begin{proof}
We first note that by Lemma \ref{imp1},  $u$ satisfies  the differential inequalities \eqref{maine} in the viscosity  sense. Therefore we may apply Theorem 1.1 in \cite{SS} and can assert  that  for some $\beta= \beta(\alpha, n, p)$,  there exists $G \in C^{\beta_1}(\overline{\partial \Omega \cap B_{r}(x_0)})$ which is the \lq\lq gradient" of $u$ at the boundary  such that
\begin{equation}\label{e4}
| u(x) -u(x_1) - \langle G(x_1), x- x_1 \rangle | \leq C|x-x_1|^{1+ \beta_1}\ \text{for all $x \in \Omega \cap B_{r}(x_0)$}.
\end{equation}
Here $C$ depends also on the $C^{2}$ character  of the domain $\Om$. 
Therefore, \eqref{e4} expresses the fact  that  $u$ is $C^{1, \beta}$ at the boundary. Now the fact that  $u  \in C^{1, \beta_2}_{loc}(\Om \cap B_{2r}(x_0))$ for some $\beta_2$ depending on $n, p$  follows from the interior  regularity  result established  in  the recent paper \cite{APR}.  At this point,  by taking $\beta= \text{min}(\beta_1, \beta_2)$,   we can argue as in the proof of Proposition 2.4 in \cite{MS} to conclude that $u \in C^{1, \beta}(\overline{\Omega \cap B_{r}(x_0)})$.
\end{proof}

%%%%%%%%%%%%%%%%%%%%%%%%%%%%%%%%%%%%%%%%%%%%%%%%%%%%%%%%%
\bigskip
{\bf Acknowledgements:}
This research was essentially done during a visit of the second author to TIFR CAM Bangalore. B.K. thanks the Institute  and in particular Agnid Banerjee for their hospitality and support. Both authors thank the anonymous referee for carefully reading the manuscript and for his/her helpful comments and suggestions.
%%%%%%%%%%%%%%%%%%%%%%%%%%%%%%%%%%%%%%%%%%%%%%%%%%%%%%%%

\medskip
\noindent
Author's addresses:

\medskip\noindent
Bernd Kawohl,
Mathematisches Institut,
Universit\"at zu K\"oln,
D-50923 K\"oln,
Germany,
kawohl@mi.uni-koeln.de

\medskip\noindent
Agnid Banerjee, TIFR CAM, Bangalore -560065, India, 
agnidban@gmail.com
\end{document}